\documentclass{amsart}

\usepackage{pstricks,color}

\usepackage{pst-plot}

\usepackage{amsmath,amsthm}     
\usepackage{amssymb}            
\usepackage{euscript}           
\usepackage{graphicx,enumerate,calc,lscape,color}
\usepackage[matrix,arrow,curve,frame]{xy}

\xymatrixcolsep{1.9pc}                          
\xymatrixrowsep{1.9pc}
\newdir{ >}{{}*!/-5pt/\dir{>}}                  

\newtheorem{thm}[subsection]{Theorem}
\newtheorem{defn}[subsection]{Definition}

\newtheorem{lemma}[subsection]{Lemma}

\newtheorem{ques}[subsection]{Problem}

\theoremstyle{definition}  

\newcommand{\F}		    {\mathbb{F}}
\newcommand{\C}         {\mathbb{C}}

\DeclareMathOperator{\Spec}{Spec}

\newcommand{\ol}{\overline}

\newcommand{\anglb}[1]{{{\langle}#1{\rangle}}}

\numberwithin{equation}{subsection}

\newpsobject{rhotower}{psline}{linecolor=red,arrows=->}
\newpsobject{hzerotower}{psline}{linecolor=black,arrows=->}
\newpsobject{honetower}{psline}{linecolor=black,arrows=->}

\newpsobject{hzero}{psline}{linecolor=black}
\newpsobject{hone}{psline}{linecolor=black}
\newpsobject{honezero}{psline}{linecolor=blue}
\newpsobject{honezerotau}{psline}{linecolor=blue,linestyle=dashed,dash=0.1 0.07}
\newpsobject{rhomult}{psline}{linecolor=red}
\newpsobject{hzerotwo}{psline}{linecolor=green}
\newpsobject{hzerotwotau}{psline}{linecolor=green,linestyle=dashed,dash=0.1 0.07}
\newpsobject{rhomulttau}{psline}{linecolor=red,linestyle=dashed,dash=0.1 0.07}

\newpsobject{rhogen}{pscircle}{linecolor=red}
\newpsobject{gen}{pscircle}{linecolor=black}

\newpsobject{taumult}{psline}{linestyle=dashed,dash=0.1 0.2}
\newgray{gridline}{0.9}
\newrgbcolor{multcolor}{0 1 1}
\newpsobject{extn}{psline}{linecolor=red}
\newpsobject{mult}{psline}{linecolor=multcolor}
\newrgbcolor{gsquare}{0.3 0.3 1}
\newrgbcolor{gcube}{1 0 1}

\begin{document}

\title{When is a fourfold Massey product defined?}

\author{Daniel C.\ Isaksen}
\address{Department of Mathematics\\ Wayne State University\\
Detroit, MI 48202}
\email{isaksen@math.wayne.edu}
\thanks{The author was supported by NSF grant DMS-1202213.}



\begin{abstract}
We define a new invariant in the homology of a differential
graded algebra.  This invariant is the obstruction to 
defining a fourfold Massey product.
\end{abstract}

\maketitle

\section{Introduction}

Massey products and Toda brackets are an essential tool for a detailed
understanding of the cohomology of the Steenrod algebra and stable homotopy groups
of spheres (see, for example, \cite{MT} \cite{BMT} \cite{BJM}).
The standard references on Massey products, such as \cite{K} \cite{M},
typically assume that brackets are strictly defined, i.e., that the subbrackets
have no indeterminacy.
We have found in our own work on motivic stable homotopy groups
\cite{I1} \cite{I2} 
that strictly defined brackets are not always general enough.

This note addresses a subtlety with the definition of fourfold Massey products,
which arises when both of the threefold subbrackets have indeterminacy.
We will define a new invariant 
(see Definition \ref{defn:coindet})
of the homology of a differential graded algebra.
Our main result (see Theorem \ref{thm:main}) is that this invariant
is the obstruction to defining a fourfold bracket whose threefold subbrackets
contain zero.

We work with a differential graded $\F_2$-algebra $A$
whose homology is $H$.
The reader who is interested in other characteristics 
can insert appropriate signs.  We suppress the grading of $A$ because it plays
no essential role here.
In general, $A$ need not be commutative.

The symbols $a_i$ always represent cycles, and the products
$a_i a_{i+1}$ are always assumed to be boundaries.  In other words,
all threefold brackets are assumed to be defined.
For any cycle $x$ in $A$, we write $\ol{x}$ for the element of $H$ 
that is represented by $x$.

\section{The problem}

Let us first recall how to compute a fourfold Massey product
$\anglb{\ol{a_0}, \ol{a_1}, \ol{a_2}, \ol{a_3}}$.
First choose elements
$a_{01}$, $a_{12}$, and $a_{23}$ such that
$d (a_{01}) = a_0 a_1$, 
$d (a_{12}) = a_1 a_2$, and
$d (a_{23}) = a_2 a_3$.
Next, choose elements $a_{02}$ and $a_{13}$ such that
$d(a_{02}) = a_0 a_{12} + a_{01} a_2$ and
$d(a_{13}) = a_1 a_{23} + a_{12} a_3$.
The bracket $\anglb{\ol{a_0}, \ol{a_1}, \ol{a_2}, \ol{a_3}}$ 
is the subset of $H$ consisting
of all elements of the form
\[
\ol{a_0 a_{13} + a_{01} a_{23} + a_{02} a_3}.
\]

There is a subtlety that arises if $\anglb{\ol{a_0}, \ol{a_1}, \ol{a_2}}$ or
$\anglb{\ol{a_1}, \ol{a_2}, \ol{a_3}}$ has indeterminacy.
In this case, one must be careful to 
choose $a_{01}$, $a_{12}$, and $a_{23}$ in such a way
that $a_0 a_{12} + a_{01} a_2$ and $a_1 a_{23} + a_{12} a_3$ are boundaries.
The well-known Lemma \ref{lem:half-strict} addresses a simple case of this
phenomenon.

\begin{lemma}
\label{lem:half-strict}
Suppose that both 
$\anglb{\ol{a_0}, \ol{a_1}, \ol{a_2}}$ and
$\anglb{\ol{a_1}, \ol{a_2}, \ol{a_3}}$ contain zero, and at least
one of the brackets is strictly zero.
Then
$\anglb{\ol{a_0}, \ol{a_1}, \ol{a_2}, \ol{a_3}}$ is defined.
\end{lemma}

\begin{proof}
Suppose that $\anglb{\ol{a_0}, \ol{a_1}, \ol{a_2}}$ is strictly zero.
Choose $a_{12}$ and $a_{23}$ such that $a_1 a_{23} + a_{12} a_3$
is a boundary.  Then any choice of $a_{01}$ makes
$a_0 a_{12} + a_{01} a_2$ into a boundary, 
since $\anglb{\ol{a_0}, \ol{a_1}, \ol{a_2}}$ is strictly zero.

The same argument applies when
$\anglb{\ol{a_1}, \ol{a_2}, \ol{a_3}}$ is strictly zero.
\end{proof}

If
both $\anglb{\ol{a_0}, \ol{a_1}, \ol{a_2}}$ and $\anglb{\ol{a_1}, \ol{a_2}, \ol{a_3}}$
have indeterminacies, it may be impossible to choose
$a_{01}$, $a_{12}$, and $a_{23}$ such that both
$a_0 a_{12} + a_{01} a_2$ and $a_1 a_{23} + a_{12} a_3$ are boundaries
simultaneously.
The problem is that there are two constraints on $a_{12}$, and it may
not be possible to satisfy both constraints.

\section{Coindeterminacy}

Suppose that $\anglb{\ol{a_0}, \ol{a_1}, \ol{a_2}}$ and 
$\anglb{\ol{a_1}, \ol{a_2}, \ol{a_3}}$ both contain zero, but both
may possibly have non-zero indeterminacy.

\begin{defn}
\label{defn:coindet}
The coindeterminacy
of the brackets
$\anglb{\ol{a_0}, \ol{a_1}, \ol{a_2}}$ and 
$\anglb{\ol{a_1}, \ol{a_2}, \ol{a_3}}$ is 
the subset of $H$ consisting of all elements
of the form $\ol{x+y}$,
where $x$ ranges over all elements of $A$ such that $d(x) = a_1 a_2$
and $a_0 x + z a_2$ is a boundary for some $z$ with $d(z) = a_0 a_1$;
and $y$ ranges over all elements of $A$ such that $d(y) = a_1 a_2$
and $a_1 w + y a_3$ is a boundary for some $w$ with $d(w) = a_2 a_3$.
\end{defn}

In other words, $x$ ranges over all possible choices of 
$a_{12}$ that can be used to construct zero in $\anglb{\ol{a_0}, \ol{a_1}, \ol{a_2}}$.
Similarly, $y$ ranges over all possible choices of
$a_{12}$ that can be used to construct zero in $\anglb{\ol{a_1}, \ol{a_2}, \ol{a_3}}$.

The careful reader can verify that the coindeterminacy is well-defined in $H$, i.e.,
\begin{enumerate}
\item
if $x$ and $y$ satisfy the conditions of Definition \ref{defn:coindet},
then $x+y$ is a cycle.
\item
if $x$ and $y$ satisfy the conditions of Definition \ref{defn:coindet}
and $b$ is a boundary, then
then $x+b$ and $y$ satisfy the conditions of Definition \ref{defn:coindet}.
\item
if $a_i'$ is homologous to $a_i$ for each $i$, then
the coindeterminacy of
$\anglb{\ol{a'_0}, \ol{a'_1}, \ol{a'_2}}$ and 
$\anglb{\ol{a'_1}, \ol{a'_2}, \ol{a'_3}}$ is the same as
the coindeterminacy of 
$\anglb{\ol{a_0}, \ol{a_1}, \ol{a_2}}$ and 
$\anglb{\ol{a_1}, \ol{a_2}, \ol{a_3}}$.
\end{enumerate}

\begin{defn}
Let $\ol{a}$ and $\ol{b}$ be elements of $H$.
Then $(\ol{a} \backslash \backslash \ol{b})$ 
is the additive subgroup of $H$ consisting of all elements $\ol{x}$
such that $\ol{a} \ol{x} = \ol{z} \ol{b}$ for some $\ol{z}$ in $H$,
and $(\ol{a} // \ol{b})$ is the additive subgroup of $H$
consisting of all elements $\ol{x}$ such that $\ol{a} \ol{z} = \ol{x} \ol{b}$ 
for some $\ol{z}$ in $H$.
\end{defn}

\begin{lemma}
The coindeterminacy of $\anglb{\ol{a_0}, \ol{a_1}, \ol{a_2}}$ and
$\anglb{\ol{a_1}, \ol{a_2}, \ol{a_3}}$ is a coset with respect to 
$(\ol{a_0} \backslash \backslash \ol{a_2}) + (\ol{a_1} // \ol{a_3})$.
\end{lemma}

One possible name for
$(\ol{a_0} \backslash \backslash \ol{a_2}) + (\ol{a_1} // \ol{a_3})$ is
the ``indeterminacy of the coindeterminacy".

\begin{proof}
Let $x+y$ and $x'+y'$ represent elements of the coindeterminacy.
We will consider $(x+y) + (x'+y') = (x+x') + (y+y')$.

The element $x+x'$ is a cycle.  There exist elements $z$ and $z'$
such that
$a_0 x + z a_2$ and $a_0 x' + z' a_2$ are boundaries.
Therefore, $a_0 (x+x')$ is homologous to $(z+z') a_2$.
This shows that $\ol{x+x'}$ belongs to $(\ol{a_0} \backslash \backslash \ol{a_2})$.
Similarly, $\ol{y+y'}$ belongs to $(\ol{a_1} // \ol{a_3})$.

On the other hand, let $\ol{x+y}$ be an element of the coindeterminacy,
and let $\ol{c}$ be an element of 
$(\ol{a_0} \backslash \backslash \ol{a_2})$.  Choose a cycle $e$ such that
$\ol{a_0} \ol{c} = \ol{e} \ol{a_2}$.  
There exists $z$ in $A$ such that $a_0 x+ z a_2$ is a boundary.
Then $a_0 (x + c) + (z + e)a_2$ is also a boundary.
This shows that $\ol{(x+c) + y}$ also belongs to the coindeterminacy.
Similarly, if $\ol{c}$ in an element of $(\ol{a_1} // \ol{a_3})$,
then $\ol{x + (y + c)}$ also belongs to the coindeterminacy.
\end{proof}

\begin{thm}
\label{thm:main}
Suppose that 
$\anglb{\ol{a_0}, \ol{a_1}, \ol{a_2}}$ and $\anglb{\ol{a_1}, \ol{a_2}, \ol{a_3}}$
both contain zero but possibly have non-zero indeterminacy.
The fourfold bracket $\anglb{\ol{a_0}, \ol{a_1}, \ol{a_2}, \ol{a_3}}$
is defined if and only if zero is contained in
the coindeterminacy of $\anglb{\ol{a_0}, \ol{a_1}, \ol{a_2}}$ and 
$\anglb{\ol{a_1}, \ol{a_2}, \ol{a_3}}$.
\end{thm}

\begin{proof}
Suppose that $\anglb{\ol{a_0}, \ol{a_1}, \ol{a_2}, \ol{a_3}}$ is defined.
There are elements $a_{01}$, $a_{12}$, and $a_{23}$ such that
$a_0 a_{12} + a_{01} a_2$ and
$a_1 a_{23} + a_{12} a_3$ are boundaries.
Then $0 = a_{12} + a_{12}$ is an element of the coindeterminacy.

Suppose that zero belongs to the coindeterminacy.
In the notation from Definition \ref{defn:coindet},
we have $x = y$.  Let $a_{01}$, $a_{12}$, and $a_{23}$ be
$z$, $x$, and $w$ respectively.
\end{proof}

\section{An Example}

\begin{defn}
Let $A$ be the differential graded algebra whose underlying
algebra is a commutative polynomial algebra on the generators listed in the
table.
The differential on $A$ is defined on generators as below, and then extended
to all of $A$ via the Leibniz rule.

\begin{center}
\begin{tabular}{ll}
$x$ & $d(x)$ \\
\hline
$a_0$ & $0$ \\
$a_1$ & $0$ \\
$a_2$ & $0$ \\
$a_3$ & $0$ \\
$a_{01}$ & $a_0 a_1$ \\
$a_{12}$ & $a_1 a_2$ \\
$a_{23}$ & $a_2 a_3$ \\
$c$ & $0$ \\
$a_{02}$ & $a_0 a_{12} + a_{01} a_2$ \\
$a_{13}$ & $a_1 a_{23} + a_{12} a_3$ \\
\hline
\end{tabular}
\end{center}
\end{defn}

Note that the indeterminacy of the subbracket $\anglb{\ol{a_0}, \ol{a_1}, \ol{a_2}}$
contains
$\ol{a_0} \ol{c}$, and the indeterminacy of the subbracket 
$\anglb{\ol{a_1}, \ol{a_2}, \ol{a_3}}$
contains $\ol{c} \ol{a_3}$.
Nevertheless, the fourfold bracket $\anglb{\ol{a_0}, \ol{a_1}, \ol{a_2}, \ol{a_3}}$
is defined
because the coindeterminacy contains zero.

\begin{defn}
Let $A'$ be the differential graded algebra whose underlying algebra
is the same as the underlying algebra of $A$.
The differential on $A'$ is the same as on $A$, except that
$d(a_{13}) = a_1 a_{23} + (a_{12} + c) a_3$.
\end{defn}

The homologies of $A$ and of $A'$ are quite similar.  They are isomorphic as
rings, and they share the same threefold Massey product structure.  However,
the bracket $\anglb{\ol{a_0}, \ol{a_1}, \ol{a_2}, \ol{a_3}}$ is not well-defined in
the homology of $A'$
because the coindeterminacy is a 
non-zero coset of $\ol{c} = \ol{a_{12} + (a_{12} + c)}$.

Therefore, coindeterminacy detects that
$A$ and $A'$ are not weakly equivalent differential graded algebras.

\section{Next steps}

We leave unanswered a number of interesting and accessible problems.

\begin{ques}
Find examples of the following phenomena in the cohomology of the Steenrod algebra:
\begin{enumerate}
\item
A fourfold bracket
that is not defined because its coindeterminacy does not contain zero.
\item
A fourfold bracket
that is defined because
the coindeterminacy contains zero, even though
both threefold brackets have non-zero indeterminacy.
\end{enumerate}
\end{ques}

\begin{ques}
Extend these ideas to higher order Massey products.
\end{ques}

\begin{ques}
Extend these results to fourfold Toda brackets.
\end{ques}

\begin{ques}
Find examples of the following phenomena in the stable homotopy groups of spheres:
\begin{enumerate}
\item
A fourfold bracket
that is not defined because its coindeterminacy does not contain zero.
\item
A fourfold bracket
that is defined because
the coindeterminacy contains zero, even though
both threefold brackets have non-zero indeterminacy.
\end{enumerate}
\end{ques}

\begin{ques}
Reinterpret coindeterminacy in terms of the existence or non-existence
of certain 5-cell complexes.
\end{ques}

\begin{ques}
Extend these ideas to higher order Toda brackets.
\end{ques}

\bibliographystyle{amsalpha}

\begin{thebibliography}{S}

\bibitem{BJM}
M.\ G.\ Barratt, J.\ D.\ S.\ Jones, and M.\ E.\ Mahowald, 
\emph{Relations amongst Toda brackets and the Kervaire invariant in dimension 62},
J.\ London Math.\ Soc.\ \textbf{30} (1984) 533-Ð550. 

\bibitem{BMT}
M.\ G.\ Barratt, M.\ E.\ Mahowald, and M.\ C.\ Tangora, 
\emph{Some differentials in the Adams spectral sequence II},
Topology \textbf{9} (1970) 309-Ð316.

\bibitem{I1}
D.\ C.\ Isaksen,
\emph{The cohomology of the motivic Steenrod algebra over $\Spec \C$},
preprint, 2013.

\bibitem{I2}
D.\ C.\ Isaksen,
\emph{Motivic stable stems}, in preparation.

\bibitem{K}
D.\ Kraines,
\emph{Massey higher products}
Trans.\ Amer.\ Math.\ Soc.\ \textbf{124} (1966) 431-Ð449. 

\bibitem{M}
J.\ P.\ May,
\emph{Matric Massey products}
J.\ Algebra \textbf{12} (1969) 533-Ð568. 

\bibitem{MT}
M.\ Mahowald, and M.\ Tangora, 
\emph{Some differentials in the Adams spectral sequence},
Topology \textbf{6} (1967) 349-Ð369.
 
\end{thebibliography}

\end{document}